\documentclass{birkmult}
\usepackage{amsmath}
\usepackage{amsthm}
\usepackage{amssymb}
\usepackage{graphicx}
\usepackage{amsfonts}
\usepackage{latexsym}
\usepackage{setspace}
\usepackage{hyperref,cite}
\usepackage{multicol}
\usepackage[active]{srcltx}
\usepackage{mathrsfs}

\newcommand{\B}{\mathcal{B}}
\newcommand{\C}{ \mathbb{C}}
\newcommand{\D}{ \mathbb{D}}

\newcommand{\ran}{\operatorname{ran}}

\newcommand{\inner}[1]{\langle #1 \rangle}

\newcommand{\M}{\mathcal{M}}

\newcommand{\T}{\mathbb{T}}
\newcommand{\h}{\mathcal{H}}
\newcommand{\K}{\mathcal{K}}
\newcommand{\minimatrix}[4]{\begin{pmatrix} #1 & #2 \\ #3 & #4 \end{pmatrix}  }

\renewcommand{\labelenumi}{(\roman{enumi})}
\renewcommand{\phi}{\varphi}
\renewcommand{\epsilon}{\varepsilon}

\theoremstyle{plain}
\newtheorem{Theorem}{Theorem}

\newtheorem{Corollary}[Theorem]{Corollary}

\newtheorem{Lemma}{Lemma}
\theoremstyle{definition}

\newtheorem{Remark}{Remark}
\newtheorem{Question}{Problem}

\allowdisplaybreaks

\begin{document}
\bibliographystyle{amsplain}

    \title{$C^{*}$-algebras generated by truncated Toeplitz operators}

    \author[S. R. Garcia]{Stephan Ramon Garcia}
    \address{   Department of Mathematics\\
            Pomona College\\
            Claremont, California\\
            91711 \\ USA}
    \email{Stephan.Garcia@pomona.edu}
    \urladdr{\url{http://pages.pomona.edu/~sg064747}}
    \thanks{The first named author was partially supported by National Science Foundation Grant DMS-1001614.}

    \author[W. T. Ross]{William T. Ross}
\address{   Department of Mathematics and Computer Science\\
            University of Richmond\\
            Richmond, Virginia\\
            23173 \\ USA}
    \email{wross@richmond.edu}
    \urladdr{\url{http://facultystaff.richmond.edu/~wross}}

\author[W. R. Wogen]{Warren R. Wogen}
	\address{Department of Mathematics, University of North Carolina, Chapel Hill, North Carolina 27599}
	\email{wrw@email.unc.edu}
	\urladdr{\url{http://www.math.unc.edu/Faculty/wrw/}}


	\dedicatory{Dedicated to the memory of William Arveson.}

    \begin{abstract}
	    We obtain an analogue of Coburn's description of the Toeplitz algebra in the setting of truncated Toeplitz operators.
	    As a byproduct, we provide several examples of complex symmetric operators
	    which are not unitarily equivalent to truncated Toeplitz operators having continuous symbols.
    \end{abstract}

\maketitle

\section{Introduction}
	In the following, we let $\h$ denote a separable complex Hilbert space and $\mathcal{B}(\h)$ denote the set of
	all bounded linear operators on $\h$.  For each $\mathscr{X} \subseteq \mathcal{B}(\h)$,
	let $C^*(\mathscr{X})$ denote the unital $C^*$-algebra generated by $\mathscr{X}$.  Since we are frequently
	interested in the case where $\mathscr{X} = \{A\}$ is a singleton, we often write $C^*(A)$ in place of $C^*(\{A\})$
	in order to simplify our notation.  
	
	Recall that the \emph{commutator ideal} $\mathscr{C}(C^*(\mathscr{X}))$ 
	of $C^*(\mathscr{X})$ is the smallest norm closed two-sided ideal which contains the \emph{commutators} 
	$[A,B]:=AB-BA$, where $A$ and $B$ range over all elements of $C^*(\mathscr{X})$.  Since
	the quotient algebra $$C^*(\mathscr{X})/ \mathscr{C}(C^*(\mathscr{X}))$$ is an abelian $C^*$-algebra,
	it is isometrically $*$-isomorphic to $C(Y)$, the set of all continuous functions on some compact Hausdorff 
	space $Y$ \cite[Thm.~1.2.1]{ConwayCOT}.  If we agree to denote isometric $*$-isomorphism by $\cong$, then we may write
	\begin{equation}\label{eq-CCCY}
		\frac{C^*(\mathscr{X})}{\mathscr{C}(C^*(\mathscr{X}))} \cong C(Y).
	\end{equation}
	This yields the short exact sequence
	\begin{equation}\label{eq-ShortExact}
		0 \longrightarrow \mathscr{C}(C^*(\mathscr{X})) \overset{\iota}{\longrightarrow} 
		C^*(\mathscr{X}) \overset{\pi}{\longrightarrow} C(Y) \longrightarrow 0,
	\end{equation}
	where $\iota: \mathscr{C}(C^*(\mathscr{X})) \to C^*(\mathscr{X})$ is the inclusion map
	and $\pi:C^*(\mathscr{X})\to C(Y)$ is the composition of the quotient map with
	the map which implements \eqref{eq-CCCY}.  

	The \emph{Toeplitz algebra} $C^*(T_z)$, where $T_z$ denotes the unilateral shift on the classical Hardy space $H^2$,
	has been extensively studied since the seminal work of Coburn in the late 1960s \cite{Coburn, Coburn2}.  
	Indeed, the Toeplitz algebra is now 
	one of the standard examples discussed in many well-known texts (e.g., \cite[Sect.~4.3]{MR1865513}, 
	\cite[Ch.~V.1]{Davidson}, \cite[Ch.~7]{MR1634900}).  In this setting, we have $\mathscr{C}(C^*(T_z)) = \mathscr{K}$,
	the ideal of compact operators on $H^2$, and $Y = \T$ (the unit circle), so that the short exact sequence
	\eqref{eq-ShortExact} takes the form
	\begin{equation}\label{eq-CoburnExact}
		0 \longrightarrow \mathscr{K} \overset{\iota}{\longrightarrow} C^*(T_z)
		\overset{\pi}{\longrightarrow} C(\T) \longrightarrow 0.
	\end{equation}
	In other words, $C^*(T_z)$ is an \emph{extension} of $\mathscr{K}$ by $C(\T)$.
	In fact, one can prove that
	\begin{equation*}
		C^*(T_z) = \{ T_{\phi} + K : \phi \in C(\T), K \in \mathscr{K}\}
	\end{equation*}
	and that each element of $C^*(T_z)$ enjoys a unique decomposition of the form
	$T_{\phi} + K$ \cite[Thm.~4.3.2]{MR1865513}.  Indeed, it is well-known that the only compact Toeplitz operator
	is the zero operator \cite[Cor.~1, p.~109]{MR1865513}.  We also note that the surjective map $\pi:C^*(T_z)\to C(\T)$ 
	in \eqref{eq-CoburnExact} is given by $\pi(T_{\phi}+K) = \phi$.
	
	The preceding results have spawned numerous generalizations and variants over the years.
	For instance, one can consider $C^*$-algebras generated by matrix-valued Toeplitz operators
	or by Toeplitz operators which act upon other Hilbert function spaces (e.g., the Bergman space \cite{Ross, ACM}).
	As another example, if $\mathscr{X}$ denotes the space of functions on $\T$ which are
	both piecewise and left continuous, then a fascinating result of Gohberg and Krupnik 
	asserts that $\mathscr{C}(C^*(\mathscr{X})) = \mathscr{K}$ and provides the short exact sequence
	\begin{equation*}
		0 \longrightarrow \mathscr{K} \overset{\iota}{\longrightarrow} 
		C^*(\mathscr{X}) \overset{\pi}{\longrightarrow} C(Y) \longrightarrow 0,
	\end{equation*}
	where $Y$ is the cylinder $\T \times [0, 1]$, endowed with a certain nonstandard topology \cite{GK}. 

	Along different lines, we seek here to replace Toeplitz operators with \emph{truncated Toeplitz operators}, a class of operators
	whose study has been largely motivated by a seminal 2007 paper of Sarason \cite{Sarason}.  Let us briefly
	recall the basic definitions which are required for this endeavor.  We refer the reader to Sarason's paper or to the recent
	survey article \cite{RPTTO} for a more thorough introduction.
	
	For each nonconstant inner function $u$, we consider the \emph{model space} 
	\begin{equation*}
		\K_u := H^2 \ominus u H^2,
	\end{equation*}
	which is simply the orthogonal complement of the standard Beurling-type subspace $uH^2$ of $H^2$.  Letting $P_u$ denote
	the orthogonal projection from $L^2 := L^2(\T)$ onto $\K_u$, for each $\phi$ in $L^{\infty}(\T)$ we define the 
	\emph{truncated Toeplitz operator} $A_{\phi}^u : \K_u\to\K_u$ by setting 
	\begin{equation*}
		A_{\phi}^uf = P_u(\phi f)
	\end{equation*}
	for $f$ in $\K_u$.
	The function $\phi$ in the preceding is referred to as the \emph{symbol} of the operator $A_{\phi}^u$.\footnote{It is possible
	to consider truncated Toeplitz operators with symbols in $L^2(\T)$, although we have little need to do so here.}  In particular, let us
	observe that $A_{\phi}^u$ is simply the compression of the standard Toeplitz operator $T_{\phi}:H^2 \to H^2$ to the
	subspace $\K_u$.  Unlike traditional Toeplitz operators, however, the symbol of a truncated Toeplitz is not unique.
	In fact, $A_{\phi}^u = 0$ if and only if $\phi$ belongs to $uH^2 + \overline{uH^2}$ \cite[Thm.~3.1]{Sarason}.
	
	In our work, the \emph{compressed shift} $A_z^u$ plays a distinguished role analogous
	to that of the unilateral shift $T_z$ in Coburn's theory.  In light of this, let us recall that the
	spectrum $\sigma(A_z^u)$ of $A_z^u$ coincides with the so-called \emph{spectrum}
	\begin{equation}\label{spec-u}
		\sigma(u) := \left\{\lambda \in \D^{-}: \liminf_{z \to \lambda} |u(z)| = 0\right\}
	\end{equation}
	of the inner function $u$ \cite[Lem.~2.5]{Sarason}.  In particular,
	if $u = b_{\Lambda} s_{\mu}$, where $b_{\Lambda}$ is a Blaschke product with zero sequence
	$\Lambda = \{\lambda_n\}$ and $s_{\mu}$ is a singular inner function with corresponding
	singular measure $\mu$, then
	\begin{equation*}
		\sigma(u) =\Lambda^{-} \cup \operatorname{supp} \mu.
	\end{equation*}
	
	With this terminology and notation in hand, we are ready to state our main result,
	which provides an analogue of Coburn's description of the Toeplitz algebra in the truncated Toeplitz setting.
	
	\begin{Theorem} \label{TheoremMainContinuous}
		If $u$ is an inner function, then
		\begin{enumerate}\addtolength{\itemsep}{0.6\baselineskip}
			\item $\mathscr{C}(C^*(A_z^u)) = \mathscr{K}$, the algebra of compact operators on $\mathcal{K}_{u}$,
			\item $C^*(A_z^u)/\mathscr{K}$ is isometrically $*$-isomorphic to $C(\sigma(u) \cap \T)$,
			\item For $\phi$ in $C(\T)$,  $A_{\phi}^u$ is compact if and only if $\phi(\sigma(u)\cap \T) = \{0\}$,
			\item $C^*(A_z^u) = \{ A_{\phi}^u  + K : \phi \in C(\T), K \in \mathscr{K}\}$,
			\item For $\phi$ in $C(\T)$,  $\sigma_{e}(A^{u}_{\phi}) = \phi(\sigma_{e}(A_z^u))$,
			\item For $\phi$ in $C(\T)$, $ \|A^{u}_{\phi}\|_{e}  = \sup\{ |\phi(\zeta)| : \zeta \in \sigma(u)\cap\T\}$,
			\item Every operator in $C^*(A_z^u)$ is of the form normal plus compact.
		\end{enumerate}
		\smallskip
		Moreover,
		\begin{equation*}
			0\longrightarrow\mathscr{K} \overset{\iota}{\longrightarrow} C^*(A_z^u) 
			\overset{\pi}{\longrightarrow} C(\sigma(u) \cap \T) \longrightarrow 0
		\end{equation*}
		is a short exact sequence and thus $C^*(A_z^u)$ is an extension of the compact operators
		by $C(\sigma(u) \cap \T)$.  In particular, the map $\pi:C^*(A_z^u) \to C(\sigma(u) \cap \T)$
		is given by
		\begin{equation*}
			\pi(A_{\phi}^u + K) = \phi|_{\sigma(u)\cap\T}.
		\end{equation*}
	\end{Theorem}
	
	The proof of Theorem \ref{TheoremMainContinuous} is somewhat involved and 
	requires a number of preliminary lemmas.  It is therefore deferred until Section \ref{SectionProof}.
	However, let us remark now that the same result holds when the hypothesis that $f$ belongs to $C(\T)$ is replaced by the weaker assumption
	that $f$ is in $\mathscr{X}_{u}$, the class of $L^{\infty}$ functions which are continuous at each point of $\sigma(u) \cap \T$. 
	In fact, given $f$ in $\mathscr{X}_{u}$, there exists a $g$ in $C(\T)$ so that 
	\begin{equation*}
		A^{u}_{f} \equiv A^{u}_{g} \pmod{\mathscr{K}}. 
	\end{equation*}
	Thus one can replace $A^{u}_{f}$ with $A^{u}_{g}$ when working
	modulo the compact operators and adapt the proof of Theorem \ref{TheoremMainContinuous}
	so that $C(\T)$ is replaced by $\mathscr{X}_{u}$. 	
		
	Using completely different language and terminology, some aspects of Theorem \ref{TheoremMainContinuous} can be proven
	by triangularizing the compressed shift $A_z^u$ according to the scheme discussed at length in \cite[Lec.~V]{N1}. 
	For instance, items (vi) and (iii) of the preceding theorem are \cite[Cor.~5.1]{AC70a} and 
	\cite[Thm.~5.4]{AC70a}, respectively (we should also mention related work of Kriete \cite{MR0328659,MR0288275}).  
	From an operator algebraic perspective, however, 
	we believe that a different approach is desirable.  Our approach is similar in spirit to the original
	work of Coburn and forms a possible blueprint for variations and extensions (see Section \ref{SectionPiecewise}).  Moreover,
	our approach does not require the detailed consideration of several special cases (i.e., Blaschke products,
	singular inner functions with purely atomic spectra, etc.) as does the approach pioneered in \cite{AC70a,MR0264386}.
	In particular, we are able to avoid the somewhat involved computations
	and integral transforms encountered in the preceding references.	
	
\section{Continuous symbols and the TTO-CSO problem}
	Recall that a bounded operator $T$ on a Hilbert space $\h$ is called
	\emph{complex symmetric} if there exists a conjugate-linear, isometric involution $J$ on $\h$ such that
	$T = JT^*J$.  It was first recognized in \cite[Prop.~3]{G-P} that every truncated Toeplitz
	operator is complex symmetric (see also \cite{CCO} where this is discussed in great detail).  
	This hidden symmetry turns out to be a crucial ingredient in Sarason's general treatment of 
	truncated Toeplitz operators \cite{Sarason}.

	A significant amount of evidence is mounting that truncated Toeplitz operators 
	may play a significant role in some sort of model theory for complex symmetric operators.  Indeed,
	a surprising and diverse array of complex symmetric operators can be concretely realized in terms of 
	truncated Toeplitz operators (or direct sums of such operators).
	The recent articles \cite{ATTO, CRW, TTOSIUES, STZ} all deal with various aspects of this problem
	and a survey of this work can be found in \cite[Sect.~9]{RPTTO}.
	
	It turns out that viewing truncated Toeplitz operators in the $C^*$-algebraic setting can shed some light on the question
	of whether every complex symmetric operator can be written in terms of truncated Toeplitz operators
	(the \emph{TTO-CSO Problem}).  
	Corollaries \ref{CorA} and \ref{CorB} below provide examples of complex symmetric operators which are not
	unitarily equivalent to truncated Toeplitz operators having continuous symbols.  To our knowledge, this is the
	first \emph{negative} evidence relevant to the TTO-CSO Problem which has been obtained.

	\begin{Corollary}\label{CorA}
		If $A$ is a noncompact operator on a Hilbert space $\h$, then the operator
		$T:\h\oplus \h \to \h \oplus \h$ defined by
		\begin{equation*}
			T = 
			\minimatrix{0}{A}{0}{0}
		\end{equation*}
		is a complex symmetric operator which is not unitarily equivalent to a truncated 
		Toeplitz operator with continuous symbol.
	\end{Corollary}
	
	\begin{proof}
		Since $T$ is nilpotent of degree two, it is complex symmetric by \cite[Thm.~2]{SNCSO}.
		However, $T$ is not of the form normal plus compact since $[T,T^*]$ is noncompact.
		Thus $T$ cannot belong to $C^*(A_z^u)$ for any $u$ by (vii) of Theorem \ref{TheoremMainContinuous}.
	\end{proof}

	\begin{Corollary}\label{CorB}
		If $S$ denotes the unilateral shift, then $T = \bigoplus_{i=1}^{\infty} (S \oplus S^*)$ is a complex symmetric operator
		which is not unitarily equivalent to a truncated Toeplitz operator with continuous symbol.
	\end{Corollary}
	
	\begin{proof}
		First note that the operator $S \oplus S^*$ is complex symmetric by \cite[Ex.~5]{G-P-II} whence $T$
		itself is complex symmetric.  Since $[S,S^*]$ has rank one, it follows that $[T,T^*]$ is noncompact.  Therefore
		$T$ is not of the form normal plus compact whence $T$ cannot belong to $C^*(A_z^u)$ for any $u$  
		by (vii) of Theorem \ref{TheoremMainContinuous}.
	\end{proof}

	Unfortunately, the preceding corollary sheds no light on the following apparently simple problem.
	
	\begin{Question}
		Is $S \oplus S^*$ unitarily equivalent to a truncated Toeplitz operator?  If so, can the symbol
		be chosen to be continuous?
	\end{Question}

\section{Proof of Theorem \ref{TheoremMainContinuous}}\label{SectionProof}

	To prove Theorem \ref{TheoremMainContinuous}, we first require a few preliminary lemmas.
	The first lemma is well-known and we refer the reader to \cite[p.~65]{N1} or \cite[p.~84]{CR} for its proof.
	
	\begin{Lemma}\label{LemmaContinuation}
		Each function in $\K_u$ can be analytically continued across $\T\!\setminus\!\sigma(u)$.
	\end{Lemma}

	The following description of the spectrum and essential spectrum of the compressed
	shift can be found in \cite[Lem.~2.5]{Sarason}, although portions of it date back to the
	work of Liv\v{s}ic and Moeller \cite[Lec.~III.1]{N1}.  The essential spectrum of $A_z^u$
	was computed in \cite[Cor.~5.1]{AC70a}.

	\begin{Lemma}\label{LemmaSpectrum}
		$\sigma(A^{u}_{z}) = \sigma(u)$ and $\sigma_{e}(A^{u}_{z}) = \sigma(u) \cap \T$. 
	\end{Lemma}

	Although the following must certainly be well-known among specialists, we do not recall 
	having seen its proof before in print.  We therefore provide a short proof of this important fact.

	\begin{Lemma}\label{LemmaIrreducible}
		$A^{u}_{z}$ is irreducible. 
	\end{Lemma}
	
	\begin{proof}
		Let $\M$ be a nonzero reducing subspace of $\K_u$ for the operator $A_z^u$.  
		In light of the fact that $\M$ is invariant 
		under the operator $I - A_z^u (A_z^u)^* = k_0 \otimes k_0$ 
		\cite[Lem.~2.4]{Sarason}, it follows that the nonzero vector 
		$k_0$ belongs to $\M$.  Since $k_0$ is a cyclic
		vector for $A_z^u$ \cite[Lem.~2.3]{Sarason}, we conclude that $\M = \K_u$.
	\end{proof}

	\begin{Lemma}\label{LemmaContinuousCompact}
		If $\phi \in C(\T)$, then $A^{u}_{\phi}$ is compact if and only if $\phi|_{\sigma(u) \cap \T} \equiv 0$. 
	\end{Lemma}

	\begin{proof}
		$(\Leftarrow)$ Suppose that $\phi|_{\sigma(u) \cap \T} \equiv 0$. Let $\epsilon > 0$ and pick 
		$\psi$ in $C(\T)$ such that $\psi$ vanishes on an open set containing $\sigma(u) \cap \T$ 
		and $\|\phi - \psi\|_{\infty} < \epsilon$.  Since $\|A^{u}_{\phi} - A^{u}_{\psi}\| \leq \|\phi - \psi\|_{\infty} < \epsilon$,
		it suffices to show that $A^{u}_{\psi}$ is compact.  To this end, we prove that if $f_{n}$ is a sequence in
		$\K_u$ which tends weakly to zero, then $A^{u}_{\psi} f_n \to 0$ in norm.
	
		Let $K$ denote the closure of $\psi^{-1}(\C\backslash\{0\})$ 
		and note that $K \subset \T \setminus\! \sigma(u)$.
		By Lemma \ref{LemmaContinuation}, we know that each $f_n$ has an 
		analytic continuation across $K$ from which it follows that $f_n(\zeta) = \inner{f_n,k_{\zeta}}\to 0$,
		where 
		$$k_{\zeta}(z) = \frac{1-\overline{u(\zeta)}u(z)}{1 - \overline{\zeta}z}$$ denotes the reproducing kernel
		corresponding to a point $\zeta$ in $K$ \cite[p.~495]{Sarason}.  Since $u$ is analytic
		on a neighborhood of the compact set $K$ we obtain
		\begin{equation*}
			|f_{n}(\zeta)|  = |\inner{ f_n, k_{\zeta}}| \leq \|f_{n}\| |u'(\zeta)|^{\frac{1}{2}} 
			\leq \sup_{n} \|f_{n}\| \sup_{\zeta \in K} |u'(\zeta)|^{\frac{1}{2}} = C < \infty
		\end{equation*}
		for each $\zeta$ in $K$.  By the dominated convergence theorem, it follows that
		\begin{equation*}
			\|A^{u}_{\psi} f_n\|^2 = \|P_{u}(\psi f_n)\| \leq \|\psi f_n\|^2 = \int_{K} |\psi|^2 |f_n|^2 \to 0
		\end{equation*}
		whence $A^{u}_{\psi} f_n$ tends to zero in norm, as desired.\medskip
	
		\noindent $(\Rightarrow)$ Suppose that $\phi$ belongs to $C(\T)$, $\xi$ belongs to $\sigma(u) \cap \T$,
		and $A^{u}_{\phi}$ is compact.  	
		Let
		$$F_{\lambda}(z) := \frac{1 - |\lambda|^2}{1 - |u(\lambda)|^2} \left|\frac{1 - \overline{u(\lambda)} u(z)}{1 - \overline{\lambda} z}\right|^2,$$
		which is the absolute value of the normalized reproducing kernel for $\mathcal{K}_{u}$.
		Observe that $F_{\lambda}(z) \geq 0$ and 
		\begin{equation*}
			\frac{1}{2 \pi} \int_{-\pi}^{\pi} F_{\lambda}(e^{it})\,dt = 1
		\end{equation*}
		by definition.

		By \eqref{spec-u}  there is sequence $\lambda_{n}$ in $\D$ such that $|u(\lambda_n)| \to 0$. 
		Suppose that $\xi = e^{i\alpha}$ and note that if $|t - \alpha| \geq \delta$, then
		\begin{equation} \label{spect-u2}
			F_{\lambda_n}(e^{it}) \leq C_{\delta} \frac{1-|\lambda_n|^2}{1 - |u(\lambda_n)|^2}\to 0.
		\end{equation}
	 This is enough to make the following approximate identity argument go through. Indeed,
		\begin{align*}
			&\left| \phi(\xi) - \frac{1}{2\pi} \int_{-\pi}^{\pi} \phi(e^{it}) F_{\lambda_n}(e^{it})\,dt \right| \\
			&\qquad \leq \frac{1}{2\pi} \int_{|t-\alpha|\leq \delta} |\phi(\xi) - \phi(e^{it})| F_{\lambda_n} (e^{it}) \,dt\\
			&\qquad \qquad + 
			\frac{1}{2\pi} \int_{|t-\alpha|\geq \delta} |\phi(\xi) - \phi(e^{it})| F_{\lambda_n} (e^{it}) \,dt.
		\end{align*}
		This first integral can be made small by the continuity of $\phi$.  Once $\delta > 0$ is fixed, the second term goes to zero by \eqref{spect-u2}.
		\end{proof}
		
		\begin{Remark}
		We would like to thank the referee for suggesting this elegant normalized kernel function proof of the $(\Rightarrow$) direction of this lemma. Our original argument was somewhat longer. 
		\end{Remark}

	\begin{Lemma}\label{LemmaContinuousCommutators}
		For each $\phi, \psi \in C(\T)$, the semicommutator $A^{u}_{\phi} A^{u}_{\psi} - A^{u}_{\phi \psi}$ is compact.
		In particular, the commutator $[A_{\phi}^u, A_{\psi}^u]$ is compact.
	\end{Lemma}

	\begin{proof}
		Let $p(z) = \sum_{i} p_i z^{i}$ and $q(z) = \sum_{j} q_j z^j$ be trigonometric polynomials on $\T$ and note that
		\begin{equation*}
			A^{u}_{p} A^{u}_{q} - A^{u}_{p q} = \sum_{i,j} p_i q_j (A^{u}_{z^i} A^{u}_{z^j} - A^{u}_{z^{i + j}}).
		\end{equation*}
		We claim that the preceding operator is compact.  Since all sums involved are finite, it suffices
		to prove that $A^{u}_{z^i} A^{u}_{z^j} - A^{u}_{z^{i + j}}$ is compact for each pair $(i,j)$ of integers.

		If $i$ and $j$ are of the same sign, then
		$A^{u}_{z^i} A^{u}_{z^j} - A^{u}_{z^{i + j}} = 0$ is trivially compact.
		If $i$ and $j$ are of different signs, then upon relabeling and taking adjoints, if necessary,
		it suffices to show that if  $n \geq m \geq 0$, then the operator 
		$A^{u}_{z^n} A^{u}_{\overline{z}^m} - A^{u}_{z^{n - m}}$ is compact (the case $n \leq m \leq 0$ being similar).  In light of the fact that
		\begin{equation*}
			A^{u}_{z^{n}} A^{u}_{\overline{z}^{m}} - A^{u}_{z^{n - m}} 
			= A^{u}_{z^{n - m}} (A^{u}_{z^m} A^{u}_{\overline{z}^m} - I),
		\end{equation*}
		we need only show that $A^{u}_{z^m} A^{u}_{\overline{z}^m}-I$ is compact for each $m \geq 1$.  However, since
		$A_z^u A_{\overline{z}}^u-I$ has rank one \cite[Lem.~2.4]{Sarason}, this follows immediately from the identity
		\begin{equation*}
			A_{z^m}^u A_{\overline{z}^m}^u  -I= 
			\sum_{\ell=0}^{m-1} A_{z^{\ell}}^u (A_z^u A_{\overline{z}}^u - I) A_{\overline{z}^{\ell}}^u.
		\end{equation*}
	
		Having shown that $A^{u}_{p} A^{u}_{q} - A^{u}_{p q}$ is compact for every pair of
		trigonometric polynomials $p$ and $q$, the desired result follows since we may uniformly approximate
		any given $\phi,\psi$ in $C(\T)$ by their respective Ces\`aro means.
	\end{proof}

	\begin{Remark}
	For Toeplitz operators, it is known that the semicommutator $T_{\phi} T_{\psi} - T_{\phi \psi}$ is compact
	under the assumption that one of the symbols is continuous, while the other belongs to $L^{\infty}$ 
	\cite[Prop.~4.3.1]{MR1865513}, \cite[Cor.~V.1.4]{Davidson}. Though not needed for the proof of our main theorem, the same is true for truncated Toeplitz operators. This was kindly pointed out to us by
	 Trieu Le.  Here is his proof: For $f$ in $L^{\infty}$, define 
		the Hankel operator $H^{u}_{f}: \mathcal{K}_{u} \to L^2$ by $H^{u}_{f} := (I - P_{u}) M_{f}$ and note that  
		$(H^{u}_{f})^{*} = P_{u} M_{\overline{f}} (I - P_u)$. For $\phi, \psi$ in $L^{\infty}$ a computation shows that 
		\begin{equation} \label{Le}	A^{u}_{\phi \psi} - A^{u}_{\phi} A^{u}_{\psi} = (H^{u}_{\overline{\phi}})^{*} H^{u}_{\psi}.
		\end{equation}
		If $f$ belongs to $L^{\infty}$, then setting $\phi = \overline{f}$ and $\phi = f$, we have 
		$$(H^{u}_{f})^{*} H^{u}_{f} = A^{u}_{\overline{f} f} - A^{u}_{\overline{f}} A^{u}_{f}.$$
		For continuous $f$ it follows from the previous Lemma that $(H^{u}_{f})^{*} H^{u}_{f}$ and hence $H^{u}_{f}$ is compact whenever $f$ is continuous. From \eqref{Le} we see that if one of $\phi$ or $\psi$ is continuous then
		$A^{u}_{\phi \psi} - A^{u}_{\phi} A^{u}_{\psi}$ is compact.
	\end{Remark}
	
	
	\begin{proof}[Proof of Theorem \ref{TheoremMainContinuous}]
		Before proceeding further, let us remark that statement (iii) has already been proven
		(see Lemma \ref{LemmaContinuousCompact}).  We first claim that 
		\begin{equation}\label{eq-ContinuousAlgebra}
			C^*(A_z^u) = C^*(\{ A_{\phi}^u : \phi \in C(\T)\}),
		\end{equation}
		noting that the containment $\subseteq$ in the preceding holds trivially.  Since
		$(A_z^u)^* = A_{\overline{z}}^u$, it follows that $A_{p}^u$ belongs to 
		$C^*(A_z^u)$ for any trigonometric polynomial $p$.  We may then uniformly approximate any
		given $\phi$ in $C(\T)$ by its Ces\`aro means to see that $A_{\phi}^u$ belongs to $C^*(A_z^u)$.
		This establishes the containment $\supseteq$ in \eqref{eq-ContinuousAlgebra}.
		
		We next prove statement (i) of Theorem \ref{TheoremMainContinuous}, 
		which states that the commutator ideal $\mathscr{C}(C^*(A_z^u))$ of $C^*(A_z^u)$ is precisely
		$\mathscr{K}$, the set of all compact operators on the model space $\K_u$:
		\begin{equation}\label{eq-CompactContainment}
			\mathscr{C}(C^*(A_z^u)) = \mathscr{K}.
		\end{equation}
		The containment $\mathscr{C}(C^*(A_z^u)) \subseteq \mathscr{K}$ follows easily from 
		\eqref{eq-ContinuousAlgebra} and Lemma \ref{LemmaContinuousCommutators}.
		On the other hand, Lemma \ref{LemmaIrreducible} tells us that $A_z^u$ is irreducible, whence
		the algebra $C^*(A_z^u)$ itself is irreducible.  Since $[A_z^u,A_{\overline{z}}^u] \neq 0$ is compact, it follows that
		$C^*(A_z^u) \cap \mathscr{K} \neq \{0\}$.  By \cite[Cor.~3.16.8]{ConwayCOT}, we conclude that 
		$\mathscr{K} \subseteq \mathscr{C}(C^*(A_z^u))$, which establishes \eqref{eq-CompactContainment}.
		
		We now claim that
		\begin{equation}\label{eq-CAZU}
			C^*(A_z^u) = \{A_{\phi}^u + K : \phi \in C(\T), K \in \mathscr{K}\},
		\end{equation}
		which is statement (iv) of Theorem \ref{TheoremMainContinuous}.  
		The containment $\subseteq$ in the preceding holds because the right-hand side
		of \eqref{eq-CAZU} is a $C^*$-algebra which contains $A_z^u$ (mimic the first portion of the proof of
		\cite[Thm.~4.3.2]{MR1865513} to see this).
		On the other hand, the containment
		$\supseteq$ in \eqref{eq-CAZU} follows because
		$C^*(A_z^u)$ contains $\mathscr{K}$ by \eqref{eq-CompactContainment} and contains
		every operator of the form $A_{\phi}^u$ with $\phi$ in $C(\T)$ by \eqref{eq-ContinuousAlgebra}.
		
		The map $\gamma: C(\T) \to C^*(A_z^u)/\mathscr{K}$ defined by 
		\begin{equation*}
			\gamma(\phi) = A^{u}_{\phi} + \mathscr{K}
		\end{equation*}
		is a homomorphism by Lemma \ref{LemmaContinuousCommutators} and hence $\gamma(C(\T))$ 
		is a dense subalgebra of $C^*(A_z^u)/\mathscr{K}$ by \eqref{eq-ContinuousAlgebra}.
		In light of Lemma \ref{LemmaContinuousCompact}, we see that
		\begin{equation}\label{eq-KernelFound}
			\ker\gamma  = \{\phi \in C(\T): \phi|_{\sigma(u) \cap \T} \equiv 0\},
		\end{equation}
		whence the map 
		\begin{equation}\label{eq-FTS01}
			\widetilde{\gamma}: C(\T)/\ker \gamma \to  C^*(A_z^u)/\mathscr{K}
		\end{equation}
		defined by
		\begin{equation*}
			\widetilde{\gamma}(\phi + \ker \gamma) = A^{u}_{\phi} + \mathscr{K}
		\end{equation*}
		is an injective $*$-homomorphism.  By \cite[Thm.~I.5.5]{Davidson}, it follows that 
		$\widetilde{\gamma}$ is an isometric $*$-isomorphism.  Since
		\begin{equation}\label{eq-FTS02}
			C(\T)/\ker \gamma \cong C(\sigma(u) \cap \T)
		\end{equation}
		by \eqref{eq-KernelFound}, it follows that
			\begin{equation*}
			\sigma_e(A_{\phi}^u) = \sigma_{C(\sigma(u)\cap\T)}(\phi) 
			= \phi( \sigma(u)\cap \T) = \phi( \sigma_e(A_{z}^u)),
		\end{equation*}
		where $\sigma_{C(\sigma(u)\cap\T)}(\phi)$ denotes the spectrum of $\phi$ as an element of the
		Banach algebra $C(\sigma(u)\cap\T)$.  This yields statement (v).		
		We also note that putting \eqref{eq-FTS01} and \eqref{eq-FTS02} together
		shows that $C^*(A_z^u)/\mathscr{K}$ is isometrically $*$-isomorphic to $C(\sigma(u) \cap \T)$,
		which is statement (ii).
		
		We now need only justify statement (vii).  To this end, recall that
		a seminal result of Clark \cite{MR0301534} asserts that for each $\alpha$ in $\T$, the operator
		\begin{equation}\label{eq-ClarkOperator}
			U_{\alpha} := A_{z}^u + \frac{\alpha}{1 - \overline{u(0)} \alpha} k_{0} \otimes Ck_0
		\end{equation}
		on $\K_u$ is a cyclic unitary operator and, moreover, that every unitary, rank-one perturbation of 
		$A_z^u$ is of the form \eqref{eq-ClarkOperator}.  A complete exposition of this important result
		can be found in the text \cite{CMR}.  Since
		\begin{equation*}
			U_{\alpha} \equiv A_z^u \pmod{\mathscr{K}},
		\end{equation*}
		it follows that
		\begin{equation}\label{eq-ClarkFunction}
			\phi(U_{\alpha}) \equiv A_{\phi}^u \pmod{\mathscr{K}}
		\end{equation}
		for every $\phi$ in $C(\T)$.  This is because the norm on $\B(\K_u)$ dominates
		the quotient norm on $\B(\K_u)/\mathscr{K}$ and since any $\phi$ in $C(\T)$
		can be uniformly approximated by trigonometric polynomials.  Since $\mathscr{K} \subseteq C^*(A_z^u)$,
		it follows that
		\begin{equation*}
			C^*( U_{\alpha} ) + \mathscr{K} = C^*(A_z^u), 
		\end{equation*}
		which yields the desired result.
	\end{proof}

\section{Piecewise continuous symbols}\label{SectionPiecewise}

	Having obtained a truncated Toeplitz analogue of Coburn's work, it is of interest
	to see if one can also obtain a truncated Toeplitz version of Gohberg and Krupnik's 
	results concerning Toeplitz operators with piecewise continuous symbols \cite{GK}. 
	Although we have not yet been able to complete this work, we have obtained a few partial results which
	are worth mentioning.

	Let $PC := PC(\T)$ denote the $*$-algebra of piecewise
	continuous functions on $\T$.  To get started, we make the simplifying assumption that $u$ is inner and that
	\begin{equation*}
		\sigma(u) \cap \T = \{1\}.
	\end{equation*}
	For instance, $u$ could be a singular inner function with a single atom at $1$ or a Blaschke product
	whose zeros accumulate only at $1$.  Let 
	\begin{equation*}
		\mathscr{A}_{PC}^u = \{ A_{\phi}^u : \phi \in PC\}
	\end{equation*}
	denote the set of all truncated Toeplitz operators on $\K_u$ having symbols in $PC$.  The following
	lemma identifies the commutator ideal of $C^*(\mathscr{A}_{PC}^u)$.

	\begin{Lemma}
		$\mathscr{C}(C^*(\mathscr{A}_{PC}^u)) = \mathscr{K}$.
	\end{Lemma}

	\begin{proof}
		Let 
		\begin{equation}\label{eq-Chi}
			\qquad\qquad\chi(e^{i \theta}) := 1 - \frac{\theta}{2 \pi}, \qquad 0 \leq \theta < 2 \pi,
		\end{equation}
		and notice that $\chi$ belongs to $PC$ and satisfies
		\begin{equation*}
			\chi_{+}(1) := \lim_{\theta \to 0} \chi(e^{i \theta}) = 1, \qquad \chi_{-}(1):= \lim_{\theta \to 2 \pi} \chi(e^{i \theta}) = 0.
		\end{equation*}
		If $\phi$ is any function in $PC$, then it follows that
		\begin{equation*}
			\phi - \phi_{+}(1) \chi - \phi_{-}(1) (1 - \chi)
		\end{equation*}
		is continuous at $1$ and assumes the value zero there. By the remarks following Theorem \ref{TheoremMainContinuous}
		in the introduction, we see that
		\begin{equation}\label{eq-PCmodCompact}
			A^{u}_{\phi} \equiv \alpha A^{u}_{\chi} + \beta I \pmod{\mathscr{K}},
		\end{equation}
		where $\alpha = \phi_+(1)  - \phi_-(1)$ and $\beta = \phi_-(1)$.  In light of \eqref{eq-PCmodCompact}
		it follows that
		\begin{equation*}
			[A_{\phi}, A_{\psi}] \equiv 0 \pmod{\mathscr{K}}
		\end{equation*}
		for any $\phi,\psi$ in $PC$ whence $\mathscr{C}(C^*(\mathscr{A}_{PC}^u)) \subseteq \mathscr{K}$.
		Since $A_z^u$ belongs to $\mathscr{A}_{PC}^u$, we conclude that $\mathscr{C}(C^*(\mathscr{A}_{PC}^u))$ 
		contains the nonzero commutator $[A_z^u,A_{\overline{z}}^u]$
		whence $C^*(\mathscr{A}_{PC}^u)$ is irreducible
		by Lemma \ref{LemmaIrreducible}.  Moreover, By \cite[Cor.~3.16.8]{ConwayCOT} we conclude that 
		$\mathscr{K}\subseteq\mathscr{C}(C^*(\mathscr{A}_{PC}^u))$ which concludes the proof.
	\end{proof}

	\begin{Lemma}
		$C^*(\mathscr{A}_{PC}^u) = C^{*}(A^{u}_{\chi}) + \mathscr{K}$.
	\end{Lemma}

	\begin{proof}
		The containment $\supseteq$ is clear from \eqref{eq-PCmodCompact}
		since $C^*(\mathscr{A}_{PC}^u))$ contains $\mathscr{K}$.  Conversely, the containment
		$\subseteq$ follows immediately from \eqref{eq-PCmodCompact}.
	\end{proof}

%

	From the discussion above and \cite[Cor.~I.5.6]{Davidson} we know that 
	\begin{equation*}
		\frac{ C^*(\mathscr{A}_{PC}^u) }{ \mathscr{C}(C^*(\mathscr{A}_{PC}^u))} 
		= \frac{C^{*}(A^{u}_{\chi}) + \mathscr{K}}{ \mathscr{K}}
		\cong \frac{C^{*}(A^{u}_{\chi})}{C^{*}(A^{u}_{\chi}) \cap \mathscr{K}}.
	\end{equation*}
	is a commutative $C^*$-algebra.  Unfortunately, we are unable to identify the algebra
	$C^*(A_{\chi}^u)$ in a more concrete manner.  This highlights the important fact that truncated
	Toeplitz operators such as $A_{\chi}^u$, whose symbols are neither analytic nor coanalytic, are 
	difficult to deal with.
	
	\begin{Question}
		Suppose that $\sigma(u) = \{1\}$.  Give a concrete description of $C^*(A_{\chi}^u)$
		where $\chi$ denotes the piecewise continuous function \eqref{eq-Chi}.
	\end{Question}
	
	\begin{Question}
		Provide an analogue of the Gohberg-Krupnik result for $\mathscr{A}_{PC}^u$.  In other words,
		give a description of $C^*(\mathscr{A}_{PC}^u)$ analogous to that of Theorem \ref{TheoremMainContinuous}.
	\end{Question}

\bibliography{CAGTTO}

\end{document}